\documentclass{amsart}

\usepackage{graphicx} 
\usepackage{amsmath}
\usepackage{tikz}
\usetikzlibrary{arrows.meta, calc, positioning, shadows, shapes.geometric}
\usepackage{amsmath, amssymb, amsthm}
\usetikzlibrary{arrows.meta}
\usetikzlibrary{decorations.markings}
\usepackage{caption}
\usepackage[utf8]{inputenc}
\usepackage{blindtext}
\usepackage{amssymb}
\usepackage{amsthm}
\usepackage[dvipsnames]{xcolor}
\usepackage{comment}
\usepackage{xcolor}
\usepackage{pgfkeys}
\usepackage{pgfplots}
\usepackage{tikz}
\usepackage[colorlinks=true,
            linkcolor=blue,
            citecolor=blue,
            urlcolor=blue]{hyperref}
\usetikzlibrary{arrows.meta, calc, decorations.pathreplacing, shapes.geometric}
\usetikzlibrary{arrows.meta,decorations.markings,angles,quotes}
\pgfplotsset{compat=1.18}
\newtheorem*{acknowledgement}{Acknowledgement}
\newtheorem{theorem}{Theorem}[section]

\newtheorem{lemma}[theorem]{Lemma}

\newtheorem{corollary}[theorem]{Corollary}
\newtheorem{proposition}[theorem]{Proposition}

\newtheorem{remark}[theorem]{Remark}

\title[Article Title]{Subgradient Methods on Manifolds with Lower Bounded Curvature}

\author{G. C. Bento$^\ast$}

\author{J. X. Cruz Neto}

\author{J. O. Lopes}

\author{I. D. L. Melo}
\address[G. C. Bento]{Institute of Mathematics and Statistics, IME, Federal University of Goi\'as, Goi\^ania, Goi\'as, Brazil}
\email{glaydston@ufg.br}

\address[J. X. Cruz Neto]{Departamento de Matem\'{a}tica, CCN, Federal University of Piau\'{\i}, Te\-re\-si\-na, Piau\'{\i}, Brazil.}
\email{jxavier@ufpi.edu.br}

\address[J. O. Lopes]{Departamento de Matem\'{a}tica, CCN, Federal University of Piau\'{\i}  Te\-re\-si\-na, Piau\'{\i}, Brazil.}
\email{jurandir@ufpi.edu.br}

\address[I. D. L. Melo]{Departamento de Matem\'{a}tica, CCN, Federal University of Piau\'{\i}  Te\-re\-si\-na, Piau\'{\i}, Brazil.}
\email{italodowell@ufpi.edu.br}




\thanks{$^\ast$ Corresponding author.}

\begin{document}

\begin{abstract}
 The subgradient method is a classical and foundational approach in non-smooth convex optimization, holds enduring importance due to its simplicity, robustness, and pivotal role as a conceptual and algorithmic point of departure have made it the backbone of many significant optimization algorithms.  Motivated by classical Euclidean results and recent advances in first-order Riemannian optimization, we study the convergence of the subgradient method on Hadamard manifolds with lower bounded curvature. Assuming a nonempty solution set and employing a corresponding non-summable diminishing step-size condition, we establish the convergence of the generated sequence $\{x^k\}$ to a minimizer whenever at least one of the following holds: (a) the sequence $\{x^k\}$ is bounded; (b) the solution set $S$ is bounded; or (c) the step-sizes are square-summable ($\sum_{k=1}^{\infty}\lambda_k^2<\infty$). Additionally, we prove that if $\operatorname{int}(S)\neq\emptyset$, the method achieves finite termination. Our main contribution provides a Riemannian counterpart to Shepilov's Euclidean analysis [Cybernetics, 12 (1976), pp. 544–548], thus complementing existing literature on convex minimization over manifolds with lower bounded curvature.
\end{abstract}

\maketitle
\section{Introduction}
The subgradient method is one of the classical algorithms for a first-order
solution of nondifferentiable convex optimization problems. For a convex function $f:\mathbb{R}^n \to \mathbb{R}$, one may select at iteration $k$ a subgradient $g^k \in \partial f(x^k)$ and perform the update
\begin{equation}\label{eq:subgradient}
x^{k+1} = x^k - \alpha_k g^k, \qquad k=0,1,2,\ldots,
\end{equation}
with a suitable step-size sequence $\{\alpha_k\}$. This classical Euclidean form was proposed by Shor in the 1960s and has been extensively studied; see, e.g.,
\cite{Bertsekas1999,ermol1967minimization, kiwiel2001convergence, nedic2001incremental, Polyak1969,  Saigal1996,shepilov1976} and references therein. The subgradient method is also a simple way to address convex feasibility problems, such as the classical Convex Feasibility Problem (CFP); see, for instance, \cite{bauschke1996projection, censor1997parallel, Deutsch1995}.

In the Riemannian setting, the CFP consists of finding a point in the intersection of finitely many geodesically convex sets, which are often given as sublevel sets of convex functions. In the manifold context, this problem is further motivated by the fact that certain feasibility problems that are nonconvex in the usual Euclidean sense may become convex after endowing the domain with a suitable Riemannian metric.
To the best of our knowledge, a first systematic analysis of a subgradient-type scheme for CFPs on manifolds was carried out by Bento and Melo \cite{BentoMelo2012}. In the aforementioned work, a subgradient algorithm is proposed and shown to converge to a feasible point under the assumption that the manifold has nonnegative sectional curvature. 
The influence of curvature was already noted in \cite{da1998geodesic,Ferreira1998}, where the authors presented a counterpart of the gradient (resp. subgradient) algorithm for optimization on manifolds, explicitly discussing the role of sectional curvature in the method’s convergence and proving convergence under the assumption of nonnegative sectional curvature. See also \cite{bento2013subgradient,BentoFerreiraMelo2017}, where subgradient-type schemes for convex minimization on Riemannian manifolds are explored and also focus on the case of nonnegative sectional curvature using comparison arguments to obtain Fejér-type estimates. 

As far as we know, the first approaches to dealing with subgradient-type schemes on Hadamard manifolds with lower bounded curvature trace back to \cite{wang2015linear, Wang2015}, which addressed the challenge of solving the feasibility problem on Riemannian manifolds with negative sectional curvature posed in \cite{BentoMelo2012}. These works have served as a reference for subsequent research on first-order methods on manifolds with lower bounded curvature; see, for example, \cite{ferreira2019gradient,Ferreira03042019,louzeiro2022projected,Wang2018} and the references therein.
Restricting ourselves to subgradient-type schemes for convex minimization on Riemannian manifolds, Wang \cite{Wang2018} was among the first to study subgradient algorithms for convex optimization on complete Riemannian manifolds whose sectional curvatures are bounded below. The analysis and convergence properties are established considering step-size rules that include the classical diminishing steps:
\[
\sum_{k}\alpha_k=\infty\quad\text{and}\quad\sum_{k}\alpha_k^2<\infty.
\]
Building on this framework, Ferreira et al. \cite{Ferreira03042019} went further by providing iteration-complexity bounds in addition to asymptotic analysis.
Despite these advances, it remains natural to ask to what extent classical Euclidean convergence mechanisms for the subgradient method can be transferred to the Riemannian setting under lower bounded curvature.

The subgradient method \eqref{eq:subgradient}, analyzed by Shepilov \cite{shepilov1976}, employs a (non-summable diminishing) step-size condition stated in terms of 
\[
\alpha_k \ge 0, \quad \lim_{k\to\infty}\alpha_k = 0, \quad
\sum_{k=1}^{\infty}\alpha_k= \infty.
\]
Assuming that the set $S$ of minimizers of $f$ is nonempty, it was proved that $d(x^k,S) \to 0$ provided at least one of the following additional conditions holds: $\{x^k\}$ is bounded; $S$ is bounded; $S$ is a linear manifold in $\mathbb{R}^n$; $n=2$. When $\sum \lambda_k^2 < \infty$ the sequence $\{x^k\}$ converges to a point in $S$. Furthermore, if $S$ has a nonempty interior, then the subgradient method terminates in a finite number of steps.

Motivated by these developments and by Shepilov's classical Euclidean analysis, we study convergence of the subgradient method on Hadamard manifolds with lower bounded curvature. Assuming the solution set $S$ is nonempty, we establish that $d(x^k,S) \to 0$  under the (non-summable diminishing) step-size condition expressed  whenever at least one of the following holds:
\begin{enumerate}
\item [a.] the sequence $\{x^k\}$ is bounded;
\item [b.] the solution set $S$ is bounded;
\end{enumerate}

Furthermore, we have the convergence of the iterates when at least one of the following holds:
\begin{enumerate}
\item [c.] $\displaystyle \sum_{k=1}^{\infty}\lambda_k^2<\infty$;
\item [d.] $\operatorname{int}(S)\neq\emptyset$, in which case the method terminates in finitely many steps.
\end{enumerate}

The main contribution of this paper is a Riemannian counterpart of Shepilov's Euclidean convergence analysis for the subgradient method on manifolds with lower bounded curvature, complementing recent advances in first-order Riemannian optimization for convex minimization and convergence analysis under lower bounded curvature.

The paper is organized as follows. In Section \ref{sec:preli} we collect fundamental definitions and auxiliary results from Riemannian geometry that are essential for our study. The section also presents elements of convex analysis on Hadamard manifolds, including an original contribution, which is an interesting result on its own and provides a useful technical tool for subsequent convergence analysis and for future first-order methods on manifolds. In  Section \ref{sec:Method}, we revisit the subgradient method and provide further results on its convergence. Finally, the last section presents the final considerations.

\section{Preliminares}
\label{sec:preli}
In this section, we present some pertinent concepts and results related to Riemannian geometry. For more details, see, for example, \cite{bento2023fenchel, carmo1992, petersen2006riemannian, sakai1996}.

Let $M$ be a complete and connected Riemannian manifold. It is known that $(M,d)$ is a complete metric space, where $d$ denotes the Riemannian distance. We denote by $T_xM$ the  tangent space of $M$ at $x\in M$ and by $TM$ the tangent bundle of $M$ set to be  those pairs $\theta=(x,v),$ $x\in M, \; v\in T_xM$.  The Riemannian metric of $M$ is denoted by $\langle  \cdot, \cdot   \rangle$, with the corresponding norm given by $\| \cdot \|$.  For $\theta=(x,v)\in TM$, $\gamma_{\theta}(\cdot)$ denotes the unique geodesic with initial conditions $\gamma_{\theta}(0)=x$ and $\gamma_{\theta}'(0)=v$.
Let us denote by $\mbox{exp}:TM\to M$, $\mbox{exp}(\theta):=\gamma_{\theta}(1)$, the exponential map. For each fixed $x\in M$, $\mbox{exp}_xv :=\exp(\theta)$ where $\theta = (x,v)$. The Levi-Civita connection associated with the Riemannian manifold $(M,{\langle} \cdot,\cdot {\rangle})$ is denoted $\nabla$. 
For a  smooth function on $M$, the metric induces its gradient denoted by $\mbox{grad} f$. The restriction of a geodesic to a closed bounded interval is called a geodesic segment. If $\gamma$ is a geodesic segment joining points $x$ and $y$ in $M$ then, for each $t\in [a,b]$, $\nabla$ induces a linear isometry, relative to ${ \langle}\cdot ,\cdot {\rangle}$, $P_{\gamma(a)\gamma(t)}:T_{\gamma(a)}M\to T_{\gamma(t)}M$, the so-called parallel transport along $\gamma$ from $\gamma(a)$ to $\gamma(t)$. In the particular case where $\gamma$ is the unique geodesic segment joining $x$ and $y$, then the parallel transport along $\gamma$ from $x$ to $y$ is denoted by $P_{xy}:T_{x}M\to T_{y}M$. 
For each $k$ positive, let us consider the Riemannian manifold $(M, \langle \,\cdot , \cdot\, \rangle_k)$ where $\langle u, v \rangle_k := \displaystyle\frac{1}{k^2}\langle u, v \rangle$ for any $u,v \in T_xM$. In particular, $||u||_k = \displaystyle\frac{1}{k}||u||$ where $|| \cdot||_k $ denotes the norm induced by the metric $\langle \,\cdot , \cdot\, \rangle_k$. These metrics are known as conformal metrics. 
We will denote the Riemannian manifold $(M, \langle \,\cdot , \cdot\, \rangle_k)$ by $M_k$ and by  $\nabla_k$ the Levi-Civita connection associated with the Riemannian manifold $M_k$. It is known that $\nabla_k = \nabla$ and, consequently, the geodesics of $M$ and $M_k$ coincide and $d_k = \displaystyle\frac{1}{k} d$ where $d$ and $d_k$ are, respectively, the distances in $M$ and $M_k$. Denoting by $\exp_x^k$ the exponential map in $M_k$, one has
$\exp_x^k v := \exp_x v$. Furthermore, the sectional curvatures of $M$ and $M_k$ satisfy the following equality
$$  K_x^k(w,\tilde{w}) = k^2K_x(w,\tilde{w}),$$
where $w,\tilde{w} \in T_xM$ are linearly independent vectors. For a unitary vector $v \in T_x M$, note that $\tilde{v}:=kv$ is a unitary vector with respect to the metric of $M_k$.

It is well known that the Poincaré hyperbolic disk $\mathbb{D}$ is a Hadamard manifold with constant sectional curvature equal to $-1$, which will be discussed later. In $\mathbb{D}$, the law of cosines for a triangle with side lengths $a, b, c$ and with angle opposite to side $a$ denoted by $\alpha$ is given by:
$$     
\cosh(a) = \cosh(b)\cosh(c) - \sinh(b)\sinh(c) \cos\alpha.                                                $$
Note that $\mathbb{D}_\kappa$ is a surface with constant sectional curvature equal to $-\kappa^2$ and, given a triangle in $\mathbb{D}$, there exists a corresponding triangle in $\mathbb{D}_\kappa$ whose angles are the same and the length of the sides are divided by $\kappa$ with respect to conformal metric. In particular, given  a triangle in $\mathbb{D}\kappa$ with side lengths $a',b',c'$ and angle opposite to $a'$ denoted by $\alpha$, one has the following:
$$     
\cosh (\kappa a') = \cosh (\kappa b')\cosh (\kappa c') - \sinh (\kappa b')\sinh ( \kappa c') \cos \alpha.
$$
As an application of the Rauch comparison theorem (see the proof of \cite[Lemma~3.1, page 259]{carmo1992}), we get the following comparison version of the law of cosines for Hadamard manifolds whose sectional curvature is bounded below and/or above. 
\begin{proposition}[Law of Cosines \cite{bento2023fenchel}]\label{prop:comparison} Let $M$ be an $n$-dimensional Hadamard manifold with
$\operatorname{sec}_M$ bounded below (resp.\ above) by $-\kappa^2$. Given a triangle in $M$ with side lengths $a,b,c$, and let $\alpha$ be the angle opposite to $a$. Then
\[
\cosh(\kappa a)
\leq \ (\text{resp. } \geq)\ 
\cosh(\kappa b)\cosh(\kappa c)
-
\sinh(\kappa b)\sinh(\kappa c)\cos\alpha.
\]
\end{proposition}
\subsection{Convex Analysis}
In this section, we introduce the basic definitions, notation, and key properties of convex functions on Riemannian manifolds. We begin by presenting Busemann functions as an important class of convex examples and highlight an explicit formula for their gradient on the Poincaré disk, which both illustrates our method and suggests directions for future work. Afterwards, we recall the notion of the subdifferential for a convex function and summarize its main properties; for further details, see \cite{azagra2005nonsmooth,BentoMelo2012, udriste1994, Wang2015}.

A function $f\colon M \to \mathbb{R}$ is said to be convex if and only if, for every geodesic segment $\gamma\colon [a,b]\to M$, the composition $f\circ \gamma\colon [a,b]\to \mathbb{R}$ is convex (in the usual sense). It is known that if \(f\colon M\to\mathbb{R}\) is convex, then \(f\) is continuous on \(M\).

An important class of convex functions on nonpositively curved spaces is provided by Busemann functions. Given \(p\in M\), $v\in T_pM$, \(\|v\|=1\),  and \(\gamma:[0,\infty)\to M\) a unit-speed geodesic starting at \(p\) with \(v=\gamma'(0)\in T_pM\), for \(x\in M\) define
\[
\psi_x(t)=d(x,\gamma(t))-t,\qquad t\ge 0.
\]
The function \(\psi_x(\cdot)\) is nonincreasing and bounded below by \(-d(x,p)\), and hence the limit
\[
B^p_v(x):=\lim_{t\to\infty}\bigl(d(x,\gamma(t))-t\bigr)
\]
exists for every \(x\in M\). This is the Busemann function associated to \(\gamma\), whose basic properties include (see \cite[Lemma 8.18]{bridson2013metric} and \cite[Lemma 3.4]{ballmann2013manifolds}):
\begin{itemize}
\item[(P1)] \(B^p_v:M\to\mathbb{R}\) is \(1\)-Lipschitz; 
\item[(P2)] \(B^p_v\) is convex, \(C^1\), and satisfies \(\|\mbox{grad} \, B^p_v(x)\|=1\) for all \(x\in M\).
\end{itemize}
Since they are convex and geometrically explicit on some important Hadamard manifolds, Busemann functions provide natural objective functions or constraint functions for examples and counterexamples in Riemannian optimization. Connections to recent optimization papers can be seen in results by Bento, Cruz Neto, and Melo, who showed that Busemann functions are particularly useful for addressing the absence of nonconstant affine functions on Hadamard manifolds of negative curvature. Their contributions highlight the practical and theoretical suitability of Busemann functions as substitutes for affine functions in formulating and analyzing optimization methods on such manifolds \cite{bento2022combinatorial,bento2023fenchel,de2024new}.  
Because of these features, Busemann functions form a rich, concrete class of convex functions that link global geometric structure with variational and optimization questions on manifolds. Next, we provide an explicit formula for the gradient of Busemann functions on the Poincaré disk, which will be useful to illustrate our method and to suggest directions for future work.

Two unit-speed geodesics \(\gamma\) and \(\alpha\) are called \emph{asymptotic} if there exists \(C>0\) such that
\[
d(\gamma(t),\alpha(t))\le C\quad\text{for all }t\ge0.
\]
Given any unit-speed geodesic \(\beta\) and any \(q\in M\), there exists a unique unit-speed geodesic \(\alpha\) with \(\alpha(0)=q\) that is asymptotic to \(\beta\) (see \cite[Definition~1.1]{eberlein1973visibility}).
A key consequence (\cite[Proposition~ 3.5]{eberlein1973visibility}) is: for each \(x\in M\) let \(\alpha\) be the unique unit-speed geodesic with \(\alpha(0)=x\) that is asymptotic to the unit-speed geodesic \(\gamma\) where $\gamma(0) = p$ and $v=\gamma'(0)$. Then
\[
\mbox{grad} \, B^p_v(x)=-\alpha'(0).
\]
In particular, setting \(v_x:=\alpha'(0)\), we have \(\|\mbox{grad} B^p_v(x)\|=1\) and the gradient direction equals the negative of the initial velocity vector of the asymptotic geodesic emanating from \(x\).
Let \(  \mathbb{D}=\{z\in\mathbb{C}: |z| <1\}\) be the unit disk endowed with the Poincaré metric
\[
\langle u,v\rangle_{p}=\frac{4\langle u,v\rangle}{(1-|p|^{2})^{2}},\qquad u,v\in T_{p}\mathbb{D},
\]
where \(\langle\cdot,\cdot\rangle\) denotes the standard inner product. We denote by $\|u\|$ the norm of the vector $u\in T_p \mathbb{D}$ in this metric, i.e., $||u|| = \frac{2|u|}{1-|p|^2}$. With this metric $\mathbb{D}$ is a Hadamard manifold with constant curvature $-1$, the Poincaré unit disk $\mathbb{D}$.
For a unit direction \(\eta\in\mathbb{C}\) and the geodesic ray \(\gamma(t)=\eta\tanh(t/2)\) emanating from the origin, the associated Busemann function admits the explicit form
\[
B^{0}_{\eta}(x)=\ln\!\left(\frac{|x-\eta|^{2}}{1-|x|^{2}}\right),\qquad x\in \mathbb{H}^{2}_{B},
\]
see, e.g., \cite[page 273]{bridson2013metric}. This closed form is convenient for explicit computations of horocycles and horoballs. We will now calculate an explicit formula for the gradient of the Busseman function $B^{0}_{\eta}$.
 Given $p \in \mathbb{D}$ and $\xi \in T_p\mathbb{D}$ with $||\xi|| = 1$. The Möbius transformation $\mathcal{M}:\mathbb{D}\to \mathbb{D}$ defined by
$$        
\mathcal{M}(z) = \displaystyle\frac{|\xi|\xi z + |\xi|^2p}{|\xi| \overline{p} \xi z + |\xi|^2}, 
$$
is an isometry such that $\mathcal{M}(0) = p$, see \cite{bento2023fenchel}. Observe that $D\mathcal{M}_{0}(\beta'(0)) = \xi$ where $\beta(t) = \tanh(t/2)$ in the complex notation. In particular, $(\mathcal{M} \circ \beta)(t)$ is the unique geodesic on the Poincaré disk such that $ (\mathcal{M} \circ \beta)(0) = p$ and $ (M \circ \beta)'(0) = \xi $, which gives a general expression for the geodesics on the Poincaré disk. 

Given $p \in \mathbb{D}$ let $c(t)$ the unique geodesic asymptotic to \(\gamma(t)=\eta\tanh(t/2)\) where $c(0) = p$ and $c'(0) = \xi$. We can write
$$   c(t) = \displaystyle\frac{\tanh(t/2)|\xi|\xi + |\xi|^2p}{|\xi| \overline{p} \xi \tanh(t/2) + |\xi|^2}.                       $$                                    
Since $c$ and $\gamma$ are asymptotic, it follows that $\displaystyle\lim_{t \to \infty} c(t) = \eta $ in the Euclidean distance thus

$$ \displaystyle\frac{|\xi|\xi + |\xi|^2p}{|\xi| \overline{p} \xi  + |\xi|^2} = \eta \implies \xi = \displaystyle\frac{|\xi|(\eta-p)}{1-\eta\overline{p}}.$$
On the other hand, $1 = ||\xi|| = \frac{2|\xi|}{1-|p|^2}$. Hence, $ \xi = \displaystyle\frac{1-|p|^2}{2}\displaystyle\frac{(\eta-p)}{1-\eta\overline{p}}  $ and $$\mbox{grad} B^{0}_{\eta}(p) = \displaystyle\frac{1-|p|^2}{2}\displaystyle\frac{(p-\eta)}{1-\eta\overline{p}}. $$ 
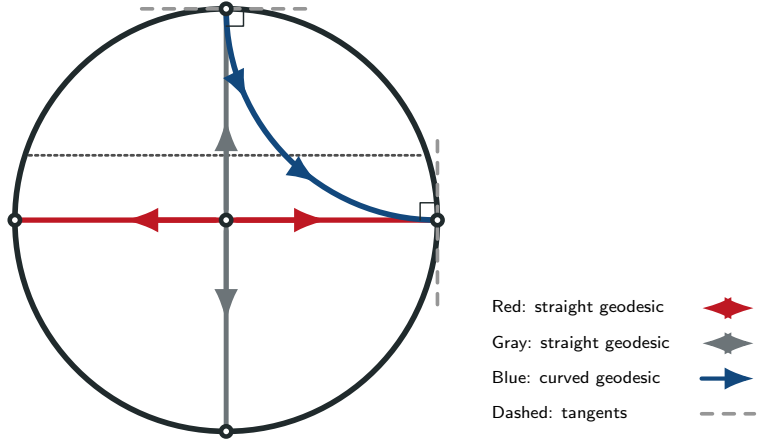
\begin{figure}[htbp]
\centering

\resizebox{0.92\textwidth}{!}{
\begin{tikzpicture}[
    >=Latex,
    line cap=round,
    line join=round,
    font=\sffamily,
    scale=0.68,
    every node/.style={transform shape}
]

\definecolor{diskborder}{RGB}{32,42,44}
\definecolor{xred}{RGB}{190,24,35}
\definecolor{ygray}{RGB}{108,116,120}
\definecolor{curveblue}{RGB}{18,72,125}
\definecolor{tangentgray}{RGB}{150,150,150}

\def\R{3.00}

\coordinate (O) at (0,0);
\coordinate (L) at (-\R,0);
\coordinate (Rgt) at (\R,0);
\coordinate (T) at (0,\R);
\coordinate (B) at (0,-\R);
\coordinate (C) at (\R,\R);


\fill[white] (-3.55,-3.35) rectangle (8.60,3.35);


\draw[
    densely dotted,
    line width=0.75pt,
    black!70
]
(-2.8,0.92) -- (2.8,0.92);


\draw[
    diskborder,
    line width=1.55pt
]
(O) circle (\R);


\draw[
    xred,
    line width=1.45pt
]
(L) -- (Rgt);

\draw[
    xred,
    line width=1.45pt,
    -{Latex[length=3.3mm,width=2.4mm]}
]
(-0.15,0) -- (-1.45,0);

\draw[
    xred,
    line width=1.45pt,
    -{Latex[length=3.3mm,width=2.4mm]}
]
(0.15,0) -- (1.45,0);


\draw[
    ygray,
    line width=1.35pt
]
(B) -- (T);

\draw[
    ygray,
    line width=1.35pt,
    -{Latex[length=3.2mm,width=2.3mm]}
]
(0,0.15) -- (0,1.45);

\draw[
    ygray,
    line width=1.35pt,
    -{Latex[length=3.2mm,width=2.3mm]}
]
(0,-0.15) -- (0,-1.45);


\draw[
    curveblue,
    line width=1.55pt,
    decoration={
        markings,
        mark=at position 0.29 with {
            \arrow{Latex[length=3.1mm,width=2.3mm]}
        },
        mark=at position 0.62 with {
            \arrow{Latex[length=3.1mm,width=2.3mm]}
        }
    },
    postaction={decorate}
]
(T)
arc[
    start angle=180,
    end angle=270,
    radius=\R
];


\draw[
    tangentgray,
    dashed,
    line width=1pt
]
(-1.2,\R) -- (1.2,\R);

\draw[
    tangentgray,
    dashed,
    line width=1pt
]
(\R,-1.2) -- (\R,1.2);


\pic[
    draw=diskborder,
    angle radius=7pt
]
{right angle = O--T--C};

\pic[
    draw=diskborder,
    angle radius=7pt
]
{right angle = O--Rgt--C};


\foreach \P in {L,Rgt,T,B,O}{
    \fill[white] (\P) circle (0.073);

    \draw[
        diskborder,
        line width=1.25pt
    ]
    (\P) circle (0.073);
}


\begin{scope}[shift={(3.65,-1.95)}]

\node[
    anchor=west,
    font=\sffamily\scriptsize
]
at (0,0.70)
{Red: straight geodesic};

\draw[
    xred,
    line width=1.15pt,
    Latex-Latex
]
(3.10,0.70) -- (3.85,0.70);

\node[
    anchor=west,
    font=\sffamily\scriptsize
]
at (0,0.20)
{Gray: straight geodesic};

\draw[
    ygray,
    line width=1.15pt,
    Latex-Latex
]
(3.10,0.20) -- (3.85,0.20);

\node[
    anchor=west,
    font=\sffamily\scriptsize
]
at (0,-0.30)
{Blue: curved geodesic};

\draw[
    curveblue,
    line width=1.15pt,
    -Latex
]
(3.10,-0.30) -- (3.85,-0.30);

\node[
    anchor=west,
    font=\sffamily\scriptsize
]
at (0,-0.80)
{Dashed: tangents};

\draw[
    tangentgray,
    dashed,
    line width=1.0pt
]
(3.10,-0.80) -- (3.85,-0.80);

\end{scope}

\end{tikzpicture}
}

\caption{ Geodesics in the Poincaré disk}

\label{fig:poincare_disk}
\label{fig:Poincare}
\end{figure}

\noindent The Figure \ref{fig:poincare_disk} represents asymptotic geodesics in the Poincaré disk. We will now study a convex function whose properties will be used for further discussion later in the paper. Consider now the geodesic rays $\alpha(t) = \tanh(-t/2)$, $\beta(t) = \tanh(t/2)$ and the associated Busemann functions $B_{\alpha}$ and $B_{\beta}$ in Poincaré disk. The function $f:\mathbb{D} \to \mathbb{R}$ defined by  $f(p) = B_{\beta}(p) +  B_{\alpha}(p)$ is convex. Note that $f(0) = 0$, for $p \neq 0$ let $\theta$ the angle between $\exp_0^{-1}(p)$ and $\beta'(0)$ and $t = d(0,p)$, where $d$ denotes the distance in the Poincaré disk. From \cite[Lemma~1]{bento2023fenchel}, it follows that $$f(p) = \ln(\cosh t - \sinh t \cos \theta) + \ln(\cosh t + \sinh t \cos \theta).$$ Thus, $f(p) = \ln(1+\sinh^2 t \sin^2 \theta) \geq 0$ and $f(p) = 0$ only if $\theta = 0$ or $\theta = \pi$, and $S:=\arg\min f$ coincide with the points of the disk on the $x$-axis. In particular, $S$ is unbounded with respect to Poincaré metric. If $p=qi$ with $-1 < q <1$ we have
\begin{eqnarray*}
\mbox{grad} f(p) = \displaystyle\frac{1-q^2}{2}\Big( \displaystyle\frac{1+qi}{1-qi} - \displaystyle\frac{1-qi}{1+qi} \Big)= \displaystyle\frac{2(1-q^2)}{1+q^2}qi.
\end{eqnarray*}
Therefore, $-\mbox{grad} \; f(p)$ always points towards the origin for points on the $y$-axis.

The following proposition is well known in the Euclidean setting. However, to the best of our knowledge, no proof in the Riemannian setting appears in the literature. For completeness, we provide below a formal statement and proof in the convex case.

\begin{proposition}[Compactness of sublevel sets]\label{prop:CSS}
Let $M$ be a Hadamard manifold and let $f:M\to\mathbb{R}$ be convex. Suppose that the solution set
\[
S:=\arg\min f=\{x\in M:\ f(x)=f^*\}, 
\qquad 
f^*:=\inf_{x\in M} f(x),
\]
is nonempty and compact. Then, for every $a\geq f^*$, the sublevel set
\[
[f\leq a]:=\{x\in M:\ f(x)\leq a\}
\]
is compact.
\end{proposition}
\begin{proof}
The convexity of $f$ implies continuity and, hence, $[f\leq a]$ is closed. It suffices to show that $[f\leq a]$ is bounded. Since the statement is trivial if \(a=f^*\), assume \(a>f^*\) and suppose, for a contradiction, that \([f\leq a]\) is unbounded. Thus, there exists $\{x^k\}\subset [f\leq a]$ such that $d(x^k,\bar{x})\to\infty$ for some $\bar{x}\in S$. Define $\ell_k:=d(x^k,\bar{x})$ and let $\gamma_k:[0,\ell_k]\to M$ be the unit-speed minimizing geodesic from $\bar{x}$ to $x^k$,
\[
\gamma_k(t)=\exp_{\bar{x}}\!\left(
t\,\frac{\exp_{\bar{x}}^{-1}(x^k)}
{\|\exp_{\bar{x}}^{-1}(x^k)\|}
\right).
\]
Let $v^k:=\exp_{\bar{x}}^{-1}(x^k)/\|\exp_{\bar{x}}^{-1}(x^k)\| \in T_{\bar{x}}M$. Since the unit sphere of $T_{\bar{x}}M$ is compact, passing to a subsequence if necessary, we may assume $v^k\to v \in T_{\bar{x}}M$ with $\|v\|=1$. By \cite[Lemma 3.2]{batista2020}, $\gamma_k(t)\to\gamma(t):=\exp_{\bar{x}}(tv)$ for every fixed $t\ge0$. For $t\ge0$ and $k$ large enough such that $t\le\ell_k$. From the convexity of $f\circ\gamma_k$ we have
\[
f(\gamma_k(t))
\le
\left(1-\frac{t}{\ell_k}\right)f(\bar{x})
+
\frac{t}{\ell_k}f(x^k)
\le
f^*+\frac{t}{\ell_k}(a-f^*).
\]
Since $\gamma_k(t) \to \gamma(t)$ and $f$ is continuous, it follows that
\[
f(\gamma(t)) = \lim_{k\to\infty} f(\gamma_k(t)) \le \lim_{k\to\infty} \left(f^* + \frac{t}{\ell_k}(a-f^*)\right) = f^*.
\] Hence $f(\gamma(t))=f^*$ for all $t\ge0$, implying $\gamma([0,\infty))\subset S$. This contradicts the compactness of $S$, because $\gamma$ is a geodesic ray. Therefore, $[f\leq a]$ is bounded and the conclusion follows from the continuity of $f$.
\end{proof}

Given a point $x \in M$, a vector $v \in T_{x}M$ is called a \emph{subgradient} of $f$ at $x$ 
if the following inequality holds:
\begin{equation}\label{subineq}
f(y) \geq f(x) + \langle v, \exp_{x}^{-1}y \rangle, \qquad  y \in M.
\end{equation}
The set of all subgradients of $f$ at $x$ is denoted by $\partial f(x)$. Since $f$ is convex, 
the set $\partial f(x)$ is nonempty, compact, and convex for every $x \in M$. When  $f$ is differentiable at $x$, the set $\partial f(x)$ reduces to a single element, namely the 
gradient vector of $f$ at $x$. Moreover, 
if $\{z^k\}$ is a bounded sequence and  $v^k \in \partial f(z^k)$ for all $k \in \mathbb{N}$, then the sequence $\{v^k\}$ is bounded.

Using the law of cosines for Hadamard manifolds with sectional curvature bounded below (see Proposition \ref{prop:comparison}), we establish the following theorem, which is essential for the proof of the main results of this paper.
\begin{theorem}\label{thm:key}
Let $f \colon M \to \mathbb{R}$ be a convex function, $x,\bar{x} \in M$ and $\delta > 0$ such that $f(u) < f(x)$ for every $u \in B[\bar{x},\delta]$. If $d(x, \bar{x}) \geq 2\delta$, then 
\begin{equation*}\label{main}
\cosh (\kappa d(z,\bar{x})) \leq  \cosh (\kappa d(x,\bar{x}))\cosh (\kappa d(z,x)) - \sinh (\kappa d(z,x))\sinh\left(\kappa\delta/2\right),
\end{equation*}
where  $\displaystyle z=\exp_{x}{(\lambda s)}$,~ $ \displaystyle s=-\frac{g}{||g||}$, ~ $g \in \partial f(x)$,~$\lambda >0$.
\end{theorem}
\begin{proof}
Let $v:= \exp_{x}^{-1}\bar{x}$ and $B := \exp_{x}^{-1}(B[\bar{x},\delta])$. Since $\exp_{x}:T_{x}M \to M$ is a diffeomorphism, there exists $\rho>0$ such that $w := v - \rho s$ is a point in the border of $B$. For $y:= \exp_{x}w$ it follows that $d(y, \bar{x}) = \delta$ and  $w=\exp^{-1}_xy$. Thus, from (\ref{subineq}) we have
\begin{eqnarray*}
f(y) \geq f(x) + \langle g, w \rangle =
 f(x) + \langle g, v \rangle - \rho \langle g, s\rangle = f(x) + \langle g, v \rangle + \rho|| g ||,
\end{eqnarray*}
where the last two equalities are obtained from the definitions of $w$ and $s$. From hypothesis, $f(y) < f(x)$. Hence,
\begin{equation}\label{ineqrho}
0 < \rho < \langle s, v \rangle.
\end{equation}
Using again the definition of $w$ and letting $\alpha:= \measuredangle(v,w)$, one has
\[
\|v\|^2 - \rho\langle s, v \rangle=\langle v, w \rangle = ||v||||w||\cos \alpha,
\]
and, consequently,
\begin{equation}\label{ineqaux}
 \rho\langle s, v \rangle  = ||v||^2 - ||v||||w|| \cos \alpha.
\end{equation}
For considering now the geodesic triangle of vertices $\bar{x}, x, y$, from Proposition \ref{prop:comparison}, we have
$$    
\cosh (\kappa \delta) \leq  \cosh (\kappa d(x,\bar{x}))\cosh (\kappa ||w||) - \sinh (\kappa d(x,\bar{x})) \sinh (\kappa ||w||) \cos{\alpha}.       
$$
In particular, 
\begin{equation}\label{eq:T}
\cos \alpha \leq T:=\displaystyle \frac{\cosh (\kappa d(x,\bar{x})) \cosh (\kappa ||w||) - \cosh (\kappa \delta)}{\sinh (\kappa d(x,\bar{x})) \sinh (\kappa ||w||)}.
\end{equation} 
From (\ref{ineqrho}), (\ref{ineqaux}) and taking into account definition of $v$, it follows that
\begin{equation}\label{desigualdade}
\langle s, v \rangle^2 \geq  d(x,\bar{x})^2 - d(x,\bar{x})||w|| T.
\end{equation}
We claim that $||w|| < d(x,\bar{x})$ and $d(x,\bar{x})^2 - d(x, \bar{x})||w|| T \geq 0 $. In fact, consider the triangle of vertices $0, w, v$ in $T_{x}M$ and let $\beta$ be the angle in the vertex $v$. By using the law of cosines and definition of $v$ and $w$, we obtain 
\begin{eqnarray*}
\cos \beta =\displaystyle\frac{\langle w - v, 0-v \rangle}{||w -v|| \cdot ||v||}=\displaystyle\frac{\langle s, v \rangle}{d(x,\bar{x})},
\end{eqnarray*}
\begin{eqnarray*}
\|w\|^2 = d(x,\bar{x})^2 + \rho^2 - 2 d(x,\bar{x})\rho \cos \beta.
\end{eqnarray*}
Combining two last equality and taking into account \eqref{ineqrho}, we have
\[
\|w\|^2=d(x,\bar{x})^2 + \rho(\rho - 2\langle s, v \rangle)
< d(x,\bar{x})^2,
\]
from which the first part of the statement follows. For proving the second part, let us  apply the triangle inequality in the geodesic triangle of vertices $\bar{x}, x, y$. It follows that
\begin{equation}\label{ImpIneq1}
0 < d(x,\bar{x}) - ||w|| \leq \delta,
\end{equation}
and, hence (for using that  $\cosh(t)$ is increasing for $t\ge  0$), it follows that:
\begin{eqnarray}\label{eq:1001}
\cosh (\kappa \delta)&\geq& \cosh (\kappa d(x,\bar{x}) - \kappa \|w\|),\nonumber\\ 
&=&
\cosh (\kappa d(x,\bar{x}))\cosh (\kappa \|w\|) 
- \sinh (\kappa d(x,\bar{x}))\sinh (\kappa \|w\|), 
\end{eqnarray}
where the last equality follows by using \(\cosh(a-b)=\cosh a\cosh b-\sinh a\sinh a\), with $a:=\kappa d(x,\bar{x})$ and $b:=\kappa\|w\|$. 
On the other hand, from the definition of $T$ in \eqref{eq:T}, we get
\begin{eqnarray*}
T &=& \displaystyle\frac{\cosh (\kappa d(x,\bar{x}))\cosh (\kappa ||w||) - \cosh (\kappa\delta)}{\sinh (\kappa d(x,\bar{x}))\sinh (\kappa ||w||)} \\
&=& 1 + \displaystyle\frac{\cosh (\kappa d(x,\bar{x}))\cosh (\kappa ||w||) - \sinh (\kappa d(x,\bar{x}))\sinh (\kappa ||w||) - \cosh (\kappa\delta)}{\sinh (\kappa d(x,\bar{x}))\sinh (\kappa ||w||)},
\end{eqnarray*}
which, combined with \eqref{eq:1001}, implies that $T\leq 1$. In particular, using again that $||w|| < d(x,\bar{x})$, it follows that $d(x,\bar{x})^2 - d(x,\bar{x})||w|| T \geq d(x,\bar{x})^2 - d(x,\bar{x})||w||\geq 0$, which proves the second part of the statement. From (\ref{desigualdade}), we obtain
\begin{equation}\label{desigualdade2}
\langle s, v \rangle \geq  \sqrt{d(x,\bar{x})^2 - d(x,\bar{x})||w|| T.}
\end{equation}
Consider now the geodesic triangle with vertices \(\bar{x},x,z\), and let \(\gamma=\measuredangle(s,v)\). Thus,
\[
\cos\gamma=\frac{\langle s,v\rangle}{d(x,\bar{x})}\geq\sqrt{1-\frac{\|w\|\,T}{d(x,\bar{x})}},
\]
where the final inequality follows from (\ref{desigualdade2}). From Proposition \ref{prop:comparison}, it follows that
\begin{equation}\label{ineq:2001}
\cosh (\kappa d(z,\bar{x}))
\leq  \cosh (\kappa d(x,\bar{x}))\cosh (\kappa d(z,x)) - \sinh (\kappa \lambda )C(x,\delta),
\end{equation}
where
\(
C(x,\delta) := \sinh (\kappa d(x,\bar{x}))\sqrt{1 - \displaystyle(||w|| T)/d(x,\bar{x})}
\). From the definition of \(T\) in \eqref{eq:T} and simple algebraic manipulations, it follows that
\begin{eqnarray}\label{formula}
    C(x,\delta) &=& \displaystyle\frac{\sinh (\kappa d(x, \bar{x}))}{\sqrt{\sinh (\kappa d(x, \bar{x}))\sinh (\kappa ||w||)}} \cdot \sqrt{\displaystyle\frac{||w||}{d(x, \bar{x})}} \cdot \sqrt{N+ \cosh (\kappa\delta)}.
    \end{eqnarray}
where  $N:= \displaystyle\frac{d(x, \bar{x})}{||w||} \sinh (\kappa d(x, \bar{x}))\sinh (\kappa ||w||) - \cosh (\kappa d(x, \bar{x}))\cosh (\kappa ||w||)$. On the other hand, note that 
\begin{small}
\begin{align*}
2N &= \frac{d(x,\bar{x})}{\|w\|}\big[\cosh(a+b)-\cosh(a-b)\big]
- \big[\cosh(a+b)+\cosh(a-b)\big] \\
&= \frac{(d(x,\bar{x})-\|w\|)\cosh(a+b)-(d(x,\bar{x})+\|w\|)\cosh(a-b)}{\|w\|},
\end{align*}
\end{small}
with \(a=\kappa d(x,\bar{x})\), \(b=\kappa\|w\|\). 
Letting $\xi = a + b$ and $\eta = a - b$, we obtain
\[
N = \frac{\eta \cosh(\xi) - \xi \cosh(\eta)}{\xi - \eta}.
\]
From the power series expansion of the hyperbolic cosine function we have
$$N = -1 + \displaystyle\frac{\eta\xi}{\xi - \eta}\Big(\displaystyle\sum_{n=1}^{\infty} \displaystyle\frac{\xi^{2n-1}}{(2n)!} - \displaystyle\frac{\eta^{2n-1}}{(2n)!}\Big).$$
Observe that $\xi > \eta > 0$ because $d(x,\bar{x}) > ||w||$. In particular, $N > -1$. From (\ref{formula}), it follows that
$$   C(x,\delta) > \displaystyle\frac{\sinh (\kappa d(x,\bar{x}))}{\sqrt{\sinh (\kappa d(x,\bar{x}))\sinh (\kappa ||w||)}} \cdot \sqrt{\displaystyle\frac{||w||}{d(x,\bar{x})}} \cdot \sqrt{\cosh (\kappa\delta) -1}.              $$
By hypothesis $d(x,\bar{x}) \geq 2\delta$. Using again  \eqref{ImpIneq1},
it follows that $||w|| \geq d(x,\bar{x}) - \delta$. Hence, one has $||w|| \geq d(x,\bar{x})/2$. On the other hand, $d(x,\bar{x}) > ||w||$. From the above inequality, we have
$$C(x,\delta) > \sqrt{\displaystyle\frac{\cosh (\kappa\delta) - 1}{2}}=\sinh\left(\kappa\delta/2\right).$$
Hence, the desired result follows from \eqref{ineq:2001}.
\end{proof}
\begin{figure}[h]
\centering


\resizebox{\textwidth}{!}{%

\begin{tikzpicture}[
    >=Latex,
    every node/.style={font=\normalsize},
    bluevec/.style={->, thick, blue!60!black},
    redvec/.style={->, thick, dashed, red!80!black},
    geod/.style={dashed, gray!80!black, thick},
    ballouter/.style={
        draw=green!50!black,
        dashed,
        thick,
        fill=green!20,
        fill opacity=0.3
    },
    ballinner/.style={
        draw=green!50!black,
        thick,
        fill=green!30,
        fill opacity=0.4
    }
]

%


\coordinate (xb) at (-2.5, -0.5);
\coordinate (x_global) at (3.5, 0);
\coordinate (z) at (1.5, -2.5);

\def\gvecX{2.5}
\def\gvecY{1.2}

\def\svecX{-2.5}
\def\svecY{-1.2}


\path[
    drop shadow={
        shadow xshift=3mm,
        shadow yshift=-3mm,
        fill=black,
        opacity=0.3
    },
    top color=white,
    bottom color=blue!15!gray!30,
    middle color=blue!5!white,
    draw=gray!50,
    thick
]
plot[
    smooth cycle,
    tension=0.7
]
coordinates {
    (-7.0, 3.5)
    (-2.0, 4.5)
    (4.0, 4.0)
    (7.5, 1.5)
    (6.0, -1.5)
    (2.0, -4.5)
    (-1.0, -4.8)
    (-5.0, -3.5)
    (-8.0, -0.5)
};

\node[font=\Large] at (-6,3.5) {$M$};


\begin{scope}[shift={(x_global)}]

\path[
    fill=blue!50,
    draw=blue!80,
    thick,
    opacity=0.9
]
(-3.0,-1.8)
-- (3.0,-0.8)
-- (4.0,1.0)
-- (-2.0,2.0)
-- cycle;

\node[
    blue!60!black,
    anchor=south
]
at (1.0,2.0)
{$T_x M$};


\draw[
    bluevec,
    very thick
]
(0,0) -- (\gvecX,\gvecY)
node[
    pos=0.6,
    above right=-2pt and -1pt,
    blue!60!black,
    fill=white,
    fill opacity=0.7,
    text opacity=1
]
{$\dfrac{g}{\|g\|}$}

node[
    pos=1.05,
    right,
    black,
    fill=white,
    fill opacity=0.7,
    text opacity=1
]
{$g \in \partial f(x)$};


\coordinate (s_tip_relative_in_plane)
at (\svecX,\svecY);

\draw[
    redvec,
    very thick
]
(0,0) -- (s_tip_relative_in_plane)

node[
    pos=0.6,
    below left=-2pt and -2pt,
    red!80!black,
    fill=white,
    fill opacity=0.7,
    text opacity=1
]
{$s$};

\node[
    red!80!black,
    font=\scriptsize,
    anchor=north east,
    fill=white,
    fill opacity=0.7,
    text opacity=1
]
at (s_tip_relative_in_plane)
{$s=-\dfrac{g}{\|g\|}$};


\fill (0,0) circle (2.5pt);

\node[
    above right=3pt of {(0,0)}
]
{$x$};

\end{scope}


\coordinate (s_global_tip_for_geodesic)
at ($(x_global)+(\svecX,\svecY)$);


\draw[ballouter]
(xb) ellipse (2.5 and 2.0);

\node[
    green!30!black
]
at (-2.5,-2.5)
{$B[\bar{x},2\delta]$};

\draw[ballinner]
(xb) ellipse (1.5 and 1.3);

\node[
    green!30!black
]
at (-2.5,-0.9)
{$B[\bar{x},\delta]$};

\node
at (-2.5,-1.4)
{$f(u)<f(x)$};


\draw[geod]
(xb)
.. controls (-1.0,1.5) and (1.5,1.0)
.. (x_global);

\node[
    fill=blue!5!white,
    fill opacity=0.8,
    text opacity=1,
    inner sep=2pt
]
at (0.3,1.0)
{$d(x,\bar{x})\geq 2\delta$};

\draw[geod]
(xb)
.. controls (-0.5,0.0) and (0.5,-1.5)
.. (z);

\node[
    fill=white,
    fill opacity=0.7,
    text opacity=1,
    inner sep=2pt
]
at (0.0,-0.9)
{$d(z,\bar{x})$};


\fill (xb) circle (2.5pt);

\node[
    left=3pt of xb,
    font=\large
]
{$\bar{x}$};

\fill (z) circle (2.5pt);

\node[
    below=3pt of z,
    yshift=-0.2cm
]
{$z=\exp_x(\lambda s)$};


\draw[
    ->,
    red!80!black,
    line width=1.2pt
]
(x_global)
.. controls ($(x_global)+(\svecX*0.4,\svecY*0.4)$)
and (2.0,-1.5)
.. (z);


\node[
    draw=gray!70,
    rounded corners=4pt,
    fill=white,
    anchor=south east,
    inner sep=5pt,
    drop shadow={
        shadow xshift=1mm,
        shadow yshift=-1mm,
        opacity=0.1
    }
]
at (8.5,-4.8)
{
\begin{tikzpicture}[
    scale=0.8,
    transform shape,
    every node/.style={
        font=\small,
        anchor=west
    }
]

\path[
    fill=blue!10,
    draw=gray!50
]
(0,0)
.. controls (0.2,0.2) and (0.4,-0.1)
.. (0.6,0)
.. controls (0.4,-0.2)
.. (0,0);

\node at (0.8,0)
{Manifold $M$};

\path[
    fill=blue!20,
    draw=blue!50,
    opacity=0.6
]
(0,-0.6)
-- (0.2,-0.4)
-- (0.6,-0.5)
-- (0.4,-0.7)
-- cycle;

\node at (0.8,-0.55)
{Tangent plane $T_x M$};

\draw[
    ->,
    red!80!black,
    thick
]
(0,-1.1) -- (0.6,-1.1);

\node at (0.8,-1.1)
{Geodesic update};

\draw[
    dashed,
    gray!80!black,
    thick
]
(0,-1.65) -- (0.6,-1.65);

\node at (0.8,-1.65)
{Riemannian distance};

\draw[
    draw=green!50!black,
    thick,
    fill=green!30,
    fill opacity=0.4
]
(0.3,-2.2) circle (0.15);

\node at (0.8,-2.2)
{Ball $B[\bar{x},\delta]$};

\draw[
    draw=green!50!black,
    dashed,
    thick,
    fill=green!20,
    fill opacity=0.3
]
(0.3,-2.8) circle (0.15);

\node at (0.8,-2.8)
{Ball $B[\bar{x},2\delta]$};

\end{tikzpicture}
};

\end{tikzpicture}%

} 

\caption{
Geometric interpretation on a Riemannian manifold.
}
\label{fig:Geometric}
\end{figure}
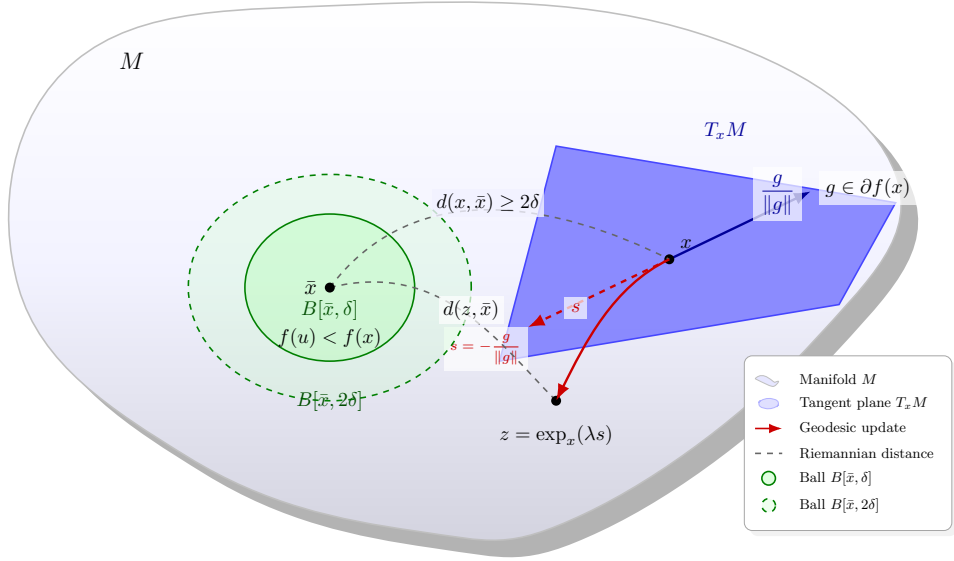

The  Figure \ref{fig:Geometric} represents the contraction property of the subgradient method on a Riemannian manifold $M$. It shows the geometric construction of a single Riemannian subgradient step, which we will formalize in the next section: starting from a point $x \in M$, a subgradient $g \in \partial f(x)$ is computed in the tangent space $T_xM$, the normalized descent direction is given by $s = -g/\|g\|$, and the next iterate is obtained as $z = \exp_x(\lambda s)$ along the geodesic determined by $s$. The point $\bar{x}$ is surrounded by a ball $B[\bar{x},\delta]$ on which $f(u) < f(x)$ for all $u \in B[\bar{x},\delta]$, and the condition $d(x,\bar{x}) \ge 2\delta$ ensures that $x$ lies outside this neighborhood. The figure highlights how the update $z$ moves along the manifold towards the region of lower function values around $\bar{x}$, illustrating the underlying contraction behavior encoded in the hyperbolic inequality for $d(z,\bar{x})$.

Throughout the remainder of the paper, unless otherwise stated, \(M\) denotes an \(n\)-dimensional Hadamard manifold with sectional curvature \(\mathrm{sec}_M \geq -\kappa^2\), and \(f\colon M \to \mathbb{R}\) denotes a convex function. With these assumptions in place, we proceed to develop the main results.

\section{Subgradient Method}\label{sec:Method}

In this section, we consider the nonsmooth optimization problem
\[
\min_{x\in M} f(x),
\]
and, motivated by recent developments and by Shepilov's classical Euclidean analysis, we study the convergence of the subgradient method on complete Riemannian manifolds with sectional curvature bounded from below. For brevity, detailed assumptions on the step-size sequence are omitted here and stated when needed.

Throughout this section, we use the following notation: \( f^{*} := \inf_{x \in M} f(x) \), \( S := \{ x \in M : f(x) = f^{*} \} \), \( B[\bar{x},\varepsilon] := \{ x \in M : d(x,\bar{x}) \leq \varepsilon \} \), \( [f \leq a] := \{ x \in M : f(x) \leq a \} \), and \( [f = a] := \{ x \in M : f(x) = a \} \).
\subsection{Method}
To define the subgradient method for minimizing $f$ on $M$, let $\{\lambda_k\}$ be a sequence of positive step sizes. The method then generates an iterative sequence $\{x^k\}$ in $M$ as follows:

\hspace{0.5cm}

\noindent\textbf{Subgradient Method (SM)}
\\[0.5ex]
\noindent\textbf{Step 0:} Choose $x^{0} \in M$ and set $k = 0$. \\[0.5ex]
\textbf{Step 1:} For the current iterate $x^{k} \in M$, compute a subgradient $g^{k} \in \partial f(x^{k})$. \\[0.5ex]
\textbf{Step 2:} If $g^{k} = 0$, \textbf{STOP}. Otherwise, define the next iterate $x^{k+1} \in M$ by
\begin{equation}\label{method}
x^{k+1} := \exp_{x^{k}} \bigl(\lambda_{k} s^{k}\bigr),
\qquad s^{k} := -\frac{g^{k}}{\|g^{k}\|}.
\end{equation}
\\[0.5ex]
\textbf{Step 2:} Update $k \gets k+1$ and return to Step 1.\\[0.5ex]

Since $f$ is a convex function, the subgradient method is well-defined. Moreover, from inequality (\ref{subineq}), it follows that 
$g^k=0$ if and only if $x^k \in S$. If there exists an iteration index  ${k}_0$ such that $g^{{k}_0}=0$, then the Subgradient Method is said to be finite. Otherwise, it is said to be infinite.
\subsection{Convergence analysis}

In this section we establish convergence results under any of the following conditions: the iterates \(\{x^k\}\) are bounded; the solution set \(S\) is bounded; \(\sum_k \lambda_k^2<\infty\); or \(S\) has nonempty interior. The choice of each case will be specified wherever required.

The following proposition is an immediate consequence of Theorem \ref{thm:key} and it plays a central role in establishing the method's convergence properties. 
\begin{proposition}\label{prop:propX} 
Given $k \in \mathbb{N}$ and $\overline{x} \in M$, assume that there exists $\delta > 0$ such that 
$f(u) < f(x^k)$ for all $u \in B[\overline{x}, \delta]$ and 
$d(x^k, \overline{x}) \geq 2\delta$. Then, for this $k$, one has
\begin{equation}\label{eq:cdelta}
\cosh (\kappa d_{k+1}) \leq  \cosh (\kappa d_k)\cosh (\kappa \lambda_k) - \sinh (\kappa \lambda_k)\sinh\left(\kappa\delta/2\right),
\end{equation}
    \begin{equation}\label{eq:cdelta2026}
    \frac{\cosh(\kappa d_{k+1}) - \cosh(\kappa d_{k})}{\sinh(\kappa \lambda_{k})} 
    \leq \cosh(\kappa d_{k}) \tanh\left(\kappa\lambda_k/2\right) - \sinh\left(\kappa\delta/2\right),
    \end{equation}
where $d_{k+1} = d(x^{k+1}, \overline{x})$, $d_{k} = d(x^{k}, \overline{x})$ and $\lambda_{k} = d(x^{k+1}, x^{k})$.   
\end{proposition}
\begin{proof}
The inequality in \eqref{eq:cdelta} is an immediate consequence of Theorem \ref{thm:key} by taking $z = x^{k+1}$, $x = x^k$, and $\lambda = \lambda_k$. 
From \eqref{eq:cdelta}, it follows that
\[
\cosh (\kappa d_{k+1}) - \cosh (\kappa d_{k}) 
\leq \cosh (\kappa d_{k})\big(\cosh (\kappa\lambda_k) - 1\big) 
- \sinh (\kappa\lambda_k)\sinh\!\left(\kappa\delta/2\right).
\]
Hence, equation \eqref{eq:cdelta2026} follows by dividing both sides by $\sinh(\kappa\lambda_k)$ and using the identity $\tanh(t/2) = (\cosh t - 1)/\sinh t$, which completes the proof.
\end{proof}
In the next two results, we assume that the sequence $\{\lambda_k\}$ satisfies  $\displaystyle\lim_{k \to \infty} \lambda_k = 0$.
\begin{corollary}\label{cor:desl1}
Let $\overline{x} \in M$, $\delta > 0$ and suppose that there exists $k_0 \in \mathbb{N}$ such that for every $k \ge k_0$,
$f(u) < f(x^k)$ for all $u \in B[\overline{x}, \delta]$ and 
$d(x^k, \overline{x}) \geq 2\delta$. If the sequence $\{x^k \}$ is bounded, then there exists $\overline{k} \in \mathbb{N}$ such that
   \[
    \cosh(\kappa d_{m+1}) \leq \cosh(\kappa d_{l}) - \sinh\left(\kappa\delta/2\right)\,\frac{\kappa}{4} \sum_{j=l}^{m} \lambda_j, 
    \quad  m > l \geq \overline{k},
    \]
 where $d_{j} = d(x^{j}, \overline{x})$ for each $j \in \mathbb{N}$.
\end{corollary}
\begin{proof}
 Since the sequence $\{x^k \}$ is bounded, there exists $B > 0$ such that $\cosh (\kappa d_k)  \leq B$ for all $k$. From \eqref{eq:cdelta}, it follows that for all $k\geq k_0$,
\begin{equation}\label{eq:cdelta1}
 \cosh (\kappa d_{k+1})\leq  \cosh (\kappa d_{k}) +B(\cosh (\kappa\lambda_k) - 1) - \sinh (\kappa\lambda_k)\sinh\left(\kappa\delta/2\right),  
\end{equation}
Using  \eqref{eq:cdelta1} for $m >l\geq k_0$, we have
\begin{equation}\label{minus}
\cosh (\kappa d_{m+1}) \leq \cosh (\kappa d_l) + B \displaystyle\sum_{j=l}^m    (\cosh (\kappa\lambda_j) - 1)  - \sinh\left(\kappa\delta/2\right)\displaystyle\sum_{j=l}^m    \sinh (\kappa\lambda_j).  
\end{equation}
On the other hand, $\displaystyle\lim_{t \to 0}\displaystyle\frac{\cosh t -1}{t^2} = 1/2$, $\displaystyle\lim_{t \to 0}\displaystyle\frac{\sinh t }{t} = 1$ and \( \displaystyle\lim_{k \to \infty} \lambda_k = 0 \). Hence, there exists $k_1 \in \mathbb{N}$ such that for $k \geq k_1$
$$ \cosh (\kappa\lambda_k) -1 < \kappa^2\lambda_{k}^2,\quad \sinh(\kappa\lambda_k) > \displaystyle\frac{\kappa\lambda_k}{2},\quad \lambda_k< \displaystyle\frac{\sinh\left(\kappa\delta/2\right)}{4\kappa B}.
$$ 
From (\ref{minus}) it follows that for $m> l \geq \bar{k}:=\max\{k_0,k_1\}$,
\begin{eqnarray*}
\cosh (\kappa d_{m+1}) &\leq& \cosh (\kappa d_l) - \displaystyle\sum_{j=l}^m   \lambda_j\left( \displaystyle\frac{\sinh\left(\kappa\delta/2\right)\kappa}{2} - B\kappa^2\lambda_j        \right),\\
&\leq& \cosh (\kappa d_l) -\displaystyle\frac{\sinh\left(\kappa\delta/2\right)\kappa }{4}\displaystyle\sum_{j=l}^m   \lambda_j, \\
\end{eqnarray*}
which completes the proof.
\end{proof}

\begin{lemma}\label{lem:limitada}
Let $\bar{x} \in M$, and let $\delta > 0$ be such that $f(u) < f(x^k)$ for every $u \in B[\bar{x}, \delta]$ and $d_k = d(x^k, \bar{x}) \geq 2\delta$ for all $k \in \mathbb{N}$. If the sequence $\{x^{k}\}$ has an accumulation point, then it is bounded.
\end{lemma}
\begin{proof}
Let $\{x^{k_j}\}$ be a convergent subsequence of $\{x^{k}\}$. Therefore, there exists $A > 0$ such that $d_{k_j} \leq A$ for all $j \in \mathbb{N}$. Set $\bar{A} := A+1$. Since $\lambda_{k} \to 0$ as $k \to \infty$, it follows that$$\lim_{k \to \infty} \cosh(\kappa \bar{A}) \tanh(\kappa \lambda_{k}/2) = 0.$$ Hence, there exists $m \in \mathbb{N}$ such that for all $k \geq m$, we have
\begin{equation}\label{desi}
 \begin{cases}
        \lambda_k < 1, \\
    \displaystyle \cosh(\kappa \bar{A}) \tanh\left(\kappa \lambda_{k}/2\right)-\sinh\left(\kappa\delta/2\right)\leq 0.
\end{cases}
\end{equation}
Now, take $k_j > m$ and $s\geq k_{{j}}$. We have two possibilities:
\begin{itemize}
    \item[(1)] $d_s \leq A < \bar{A} $,
    \item [(2)] $d_s > A $.
\end{itemize}
If case (2) holds, define $r:=\max\{n:~n<s~\mbox{and}~ d_n \le A\}$. Note that $r$ is well-defined, $r\geq k_{{j}}$, and observe that 
$$d_{r+1}\leq d(x^{r+1},x^r)+ d_r= \lambda_r+ d_r <1+A= \bar{A},$$
where the last inequality follows for combining the definition of $r$ with $\eqref{desi}$.
Combining the inequality \eqref{eq:cdelta2026} with (\ref{desi}), it follows that $d_{r+2} \leq d_{r+1
}< \bar{A}$.  Proceeding in a similar way, it follows that
$$d_{s} \leq d_{s-1}\leq \cdots \leq d_{r+2} \leq d_{r+1} \leq \bar{A}.
$$
Therefore, $d_s \le \bar{A}$, for any $s \ge k_j$, from which the conclusion of the proof follows.
\end{proof}
From this point onward, for our convergence analysis, we assume that the sequence $\{\lambda_k\}$ satisfies:
\begin{equation*}
\displaystyle\lim_{k \to \infty} \lambda_k = 0 \quad \text{and} \quad \displaystyle\sum_{k=1}^{\infty} \lambda_k = \infty.
\end{equation*}
The following result provides a condition under which the method has finite termination.
\begin{theorem}\label{interior}
Assume that $\text{int}(S)\neq \emptyset$ (the set $S$ has nonempty interior). If the sequence $\{x^{k}\}$ is bounded, then the subgradient method is finite.
\end{theorem}
\begin{proof} Let us suppose that the method is not finite, i.e., $x^k \notin S$ for all $k$. Choose $x^* \in \text{int}(S)$ and $\delta >0$ such that $ B[x^*,2\delta] \subset S$. Consequently, $f(u) < f(x^k)$ for all $u \in B[x^*,\delta]$ and $d(x^k, x^*) \geq 2\delta$ for all $k$. By Corollary \ref{cor:desl1}, there exists $\bar{k}$ such that, for all $m > l \geq \bar{k}$, we have
 \[
    \cosh(\kappa d_{m+1}) \leq \cosh(\kappa d_{l}) - \sinh\left(\kappa\delta/2\right)\,\frac{\kappa}{4} \sum_{j=l}^{m} \lambda_j, 
    \quad m > l \geq \overline{k}.
    \]
Taking the limit as $m \to \infty$ yields a contradiction, since the sequence $\{x^k\}$ is bounded whereas
$\sum_{j=l}^{\infty} \lambda_j = \infty$. Thus, the proof is complete.
\end{proof}
From now on, we assume that the subgradient method (SM) does not terminate in a finite number of steps, i.e.,\ \(x^k\notin S\) for all \(k\). The next result shows, in particular, that every accumulation point of $\{x^k\}$ is a solution of the optimization problem. 

\begin{theorem}\label{thm:liminf}
Assume that the sequence $\{x^{k}\}$ is bounded. Then, the following statements hold:
\begin{itemize}
    \item[(i)]\label{item:inf} $\displaystyle \liminf_{k \to \infty} f(x^k) = f^{*};$
    \item[(ii)]\label{item:S} $S\neq\emptyset$ and $d(x^k,S) \to 0$.
\end{itemize}
\end{theorem}
\begin{proof} For item $(i)$, suppose that there exists $\tilde{f}$ such that $f(x^k) > \tilde{f} > f^{*}$ for all $k$.
Consequently, there exists an element $\bar{x} \in M$ such that $f^*\leq f(\overline{x}) < \tilde{f}< f(x^k)$. 
By the continuity of $f$, there exists $\delta > 0$ such that  $f(u) < \tilde{f} < f(x^k)$ for all  $u \in B[\overline{x}, 2\delta]$. In particular, $d(x^k,\bar{x}) \geq 2\delta$ for all $k$. From  Corollary \ref{cor:desl1} it follows that there exist $\bar{k}$ such that for $m> l\geq \bar{k}$, we get
 \[
    \cosh(\kappa d_{m+1}) \leq \cosh(\kappa d_{l}) - \sinh\left(\kappa\delta/2\right)\,\frac{\kappa}{4} \sum_{j=l}^{m} \lambda_j, 
    \quad m > l \geq \overline{k}.
    \]
Letting $m \to \infty$, we obtain a contradiction, which concludes the proof of item $(i)$. To prove item $(ii)$, first note that from item $(i)$, we have $\liminf_{k \to \infty} f(x^{k}) = f^*$. Hence, there exists a subsequence $\{x^{k_j}\}$ such that $f(x^{k_j}) \to f^*$ as $j \to \infty$. Since $\{x^k\}$ is bounded, $\{x^{k_j}\}$ is also bounded. Thus, by passing to a further subsequence if necessary, we may assume that $x^{k_j} \to x^*$. By the continuity of $f$, we obtain
$$f({x}^*)= \lim_{j \to \infty} f(x^{k_j})= f^{*},$$
and hence $x^* \in S$. 

To prove that $d(x^k, S) \to 0$, take $K > 0$ such that $x^k \in B[x^*, K]$ for all $k \in \mathbb{N}$. Let $C:=S\cap B[x^*,K]$ and note that $C$ is a nonempty, compact and convex set. Given $\varepsilon >0$, define 
$$C_{\varepsilon}:=\left\{x \in M: d(x,C)< \frac{\varepsilon}{2}\right\}.$$
We claim that there exists $ a > f^*$ such that $[f\leq a]\cap B[x^*,K] \subset C_{\varepsilon} $. Indeed, if this were not the case, there would exist a sequence $\{y^k\}$ with $f(y^k) \to f^*$ such that $y^k \in B[x^*,K]$ and $\displaystyle d(y^k,C)\geq \frac{\varepsilon}{2}$. As  $\{y^k\}$ is bounded, it has an accumulation point $y^*$ such that  $y^* \in S\cap B[x^*,K]=C$. However, $\displaystyle d(y^*,C)\geq \frac{\varepsilon}{2}$, which is a contradiction, and the statement follows.

On the other hand, note that $C$ and $[f=a]\cap B[x^*, \overline{K}]$ are compact disjoint sets, where $\overline{K}:= K+\text{diam}(C)$ and $\text{diam}(C)$ denotes the diameter of the set $C$. Thus, there exists $\delta >0$ such that for all $\bar{x} \in C$ 
\begin{equation*}\label{eq:Bola}
B[\bar{x},2\delta]\cap\bigg([f=a]\cap B[x^*, \overline{K}]\bigg)=\emptyset.
\end{equation*}
We can assume that $2\delta <K$, from which we have $B[\bar{x},2\delta] \subset B[x^*,\overline{K}]$ because $x^{*} \in C$. Since the convexity of $f$ implies its continuity on $M$, it follows that
\begin{equation}\label{eq:conv}
B[\bar{x},2\delta] \subset [f < a],\qquad\bar{x} \in C.
\end{equation}
Since $f(x^{k_j}) \to f^*$  there exists $k_{\bar{j}}$ such that
\begin{equation}\label{Limita1}
 \begin{cases}
    f(x^{k_j}) \leq a & \text{if}~k_j \geq k_{\bar{j}},\\
    \lambda_s < \frac{\varepsilon}{2} & \text{if}~{s} \geq {k_{\bar{j}}}, \\
    \displaystyle \cosh(\kappa \overline{K})\tanh{\left(\kappa \lambda_{s}/2\right)}-\sinh\left(\kappa\delta/2\right)\leq 0 & \text{if}~{s} \geq  k_{\bar{j}}.
\end{cases}
\end{equation}
Let us study the terms of the sequence for $s\geq k_{\bar{j}}$. We have two possibilities:
\begin{itemize}
    \item[(1)] $x^s \in C_{\varepsilon}$,
    \item[(2)] $x^s \notin C_{\varepsilon}$.
\end{itemize}
If the case $(1)$ holds, we have $\displaystyle d(x^s, C)< \frac{\varepsilon}{2}$. Thus, $\displaystyle d(x^s,S) \leq d(x^s, C) <\frac{\varepsilon}{2}$.
On the other hand, if the case $(2)$ holds, let us define $r:=\max\{n:~n<s~\mbox{and}~x^n \in C_{\varepsilon}\}$. Note that $r$ is well-defined and $r\geq k_{\bar{j}}$, as well as $x^{k_{\bar{j}}} \in C_{\varepsilon}$ because $x^{k_{\bar{j}}} \in [f\leq a]\cap B[x^*,K] \subset C_{\varepsilon}$. Given that $x^{r+1}, \cdots, x^{s-1} \notin C_{\varepsilon}$ (this is a consequence of the definition of $r$), for $ k = r+1, \ldots, s-1$, one has $f(x^k)> a$.

Since $C$ is a compact and convex set, there exists a unique point $\bar{z} \in C$ such that $d(x^r,\bar{z})=d(x^r,C) < \displaystyle\frac{\epsilon}{2}$. From \eqref{eq:conv}, for $k = r+1, \ldots, s-1$, we have
\begin{equation}\label{Limita2a}
 \begin{cases}
    d(\bar{z}, x^{k}) >2\delta \\
   f(u) <a <f(x^k), & \text{for}~{u \in B[\bar{z},\delta]}. \\
     
\end{cases}
\end{equation}
Observe that, since $d(x^{r+1},x^{*}) \leq K$ and $d(x^{*},\bar{z}) \leq \operatorname{diam}(C)$, we obtain 
$$          d(x^{r+1},\bar{z}) \leq d(x^{r+1},x^{*}) + d(x^{*},\bar{z}) \le \overline{K}.                          $$
Hence, from \eqref{Limita2a} and \eqref{eq:cdelta2026} we have
\begin{eqnarray*}
    \frac{\cosh{(\kappa d(x^{r+2},\bar{z}) )}-\cosh{(\kappa d(x^{r+1},\bar{z}) )}}{\sinh{(\kappa \lambda_{r+1}})} 
    &\le& \cosh{(\kappa \overline{K})}\tanh\left(\kappa \lambda_{r+1}/2\right) -\sinh\left(\kappa\delta/2\right)\\
    &  \leq & 0,
\end{eqnarray*}
where the last inequality follows from \eqref{Limita1}. Thus $d(x^{r+2},\bar{z})\ \leq d(x^{r+1},\bar{z}) \le \overline{K}$. Proceeding in a similar way, it follows that
$$d(x^{s},\bar{z})\leq d(x^{s-1},\bar{z})\leq \cdots \leq d(x^{r+2},\bar{z}) \leq d(x^{r+1},\bar{z}).
$$
On the other hand, using the triangle inequality and \eqref{Limita1}, we have
\[
d(x^{r+1},\bar{z}) \leq d(x^{r+1},x^r) + d(x^{r},\bar{z})
\leq \lambda_r + \frac{\varepsilon}{2} < \varepsilon.
\]
Hence, $d(x^s,\bar{z}) < \varepsilon$ for all $s \geq k_{\bar{j}}$. Moreover, for all $s \geq k_{\bar{j}}$, it follows that
\[
d(x^s,S) \leq d(x^s,C) \leq d(x^{s}, \bar{z}) < \varepsilon.
\]
This completes the proof.
\end{proof}
\begin{remark}\;
\begin{itemize}
\item [(i)] Unlike the proofs in \cite{Ferreira03042019} and \cite{Wang2018}, which establish Theorem \ref{thm:liminf}(i) under the assumption that the sequence of stepsizes have square-summable and hence derive boundedness of $\{x^k\}$ from a quasi-Fejér convergence, in our approach the boundedness is instead obtained through Lemma \ref{lem:limitada}, and this is precisely where the accumulation-point hypothesis is used; see \cite[Theorem~4.2]{alber1997minimization} for a version in the linear setting, where the authors also establish the content of Theorem\ref{thm:liminf} by assuming directly that the sequence \(\{x^k\}\) is bounded instead of requiring the sequence of stepsizes to be square-summable.
\item [(ii)] 
It is possible to establish the asymptotic convergence of the method by assuming
\[
\Omega:=\{z\in M:\ f(z)\le \liminf_{k\to\infty} f(x_k)\}\neq\emptyset,
\]
instead of assuming that \(S\neq\emptyset\). However, from Theorem \ref{thm:liminf}(ii) it follows that the nonemptiness of \(\Omega\) is not an independent assumption: it is equivalent to the nonemptiness of the solution set \(S\), for example, under the assumption that the sequence of stepsizes is square-summable; see, for instance, \cite{alber1998projected, bento2018weighting,da2013subgradient};
\item [iii)] The Theorem \ref{thm:liminf}$(ii)$ has appeared in the linear setting, for example, in \cite{shepilov1976}, but under the additional assumption that $S\neq\emptyset$. 
\end{itemize}
\end{remark}

\begin{theorem}\label{thm:cluster}
Assume that $S$ is nonempty and bounded. If the sequence $\{x^k\}$ has an accumulation point,
then it is bounded.
\end{theorem}
\begin{proof}
Fix $\bar{x} \in S$ and choose $a \in \mathbb{R}$ with $a > f^*$. Since $\{x^k\}$ has an accumulation point, it possesses a convergent subsequence, which we denote by $\{x^{k_j}\}$. Thus, there exists a constant $K>0$ such that $d(x^{k_j},\bar{x})\le K$ for all $j\in\mathbb{N}$. Without loss of generality (and by adjusting the constant $a$ if necessary), we may assume that $x^{k_j}\in [f < a]$ for all $j\in\mathbb{N}$.
Since the convexity of $f$ implies its continuity on $M$, there exists $\delta>0$ such that
\begin{equation}\label{inclusao}
B[\bar{x}, 2\delta] \subset [f \leq a ].
\end{equation}
Since $S$ is compact, Proposition \ref{prop:CSS} implies that $[f \leq a]$ is compact. By increasing the constant $K$, if necessary, we may assume that
\begin{equation}\label{ineq:1011}
d(y, \bar{x}) \leq K, \quad  y \in [f \leq a ].
\end{equation}
Since $\lambda_k \to 0$, we have $\displaystyle \lim_{k \to \infty}\cosh(\kappa \overline{K})\tanh\left(\kappa \lambda_{k}/2\right)=0,$ where $\overline{K}:=K+1$. Consequently, there exists $k_0 \in \mathbb{N}$ such that, for all $k\geq k_0$, the following hold:
\begin{equation}\label{Limita2}
 \begin{cases}
        \lambda_k < 1 & \text{if}~{k} \geq k_0, \\
    \displaystyle \cosh(\kappa \overline{K})\tanh\left(\kappa \lambda_{k}/2\right)- \sinh\left(\kappa\delta/2\right)\leq 0 & \text{if}~{k} \geq k_0.
\end{cases}
\end{equation}
Now, take $k_j >k_0$ and $s\geq k_{{j}}$. We have two possibilities:
\begin{itemize}
    \item[(1)] $x^s \in [f \leq a] $,
    \item [(2)] $x^s \notin [f \leq a]$.
\end{itemize}
If (1) holds, then $d(x^s,\bar{x})\leq K$. On the other hand, if  (2) holds, define $r:=\max\{n:~n<s~\mbox{and}~x^n \in [f \leq a]\}$. Note that $r$ is well-defined, $r\geq k_{{j}}$ and $f(x^{i}) > a$ for $i=r+1, \ldots s$. Furthermore, the inclusion in \eqref{inclusao} implies that $d(x^i, \bar{x}) > 2\delta$ for $i=r+1, \ldots s$. Combining the triangle inequality with \eqref{ineq:1011} and \eqref{Limita2}, we obtain
$$d(x^{r+1},\bar{x})\leq d(x^{r+1},x^r)+ d(x^{r},\bar{x})= \lambda_r+ d(x^{r},\bar{x})< 1+K=\overline{K}.$$
It follows from \eqref{eq:cdelta2026} and (\ref{Limita2}) that $d(x^{r+2},\bar{x})\ \leq d(x^{r+1},\bar{x}) < \overline{K}$. Proceeding in a similar way, it follows that
$$d(x^{s},\bar{x})\leq d(x^{s-1},\bar{x})\leq \cdots \leq d(x^{r+2},\bar{x}) \leq d(x^{r+1},\bar{x}) \le \overline{K}.
$$
Therefore, $d(x^s,\bar{x})\leq \overline{K}$ for all $s \geq k_j$, which completes the proof of the theorem.
\end{proof}
\begin{corollary}\label{boundedS}
Assume that $S$ is nonempty and bounded. Then, either $d(x^k,\tilde{x}) \to \infty$ for some fixed $\tilde{x} \in M$, or $d(x^k, S) \to 0$.
\end{corollary}
\begin{proof} 
Since $S$ is nonempty and bounded, we have two cases to analyze regarding the sequence $\{f(x^k)\}$:
\begin{itemize}
    \item[(i)] $|f(x^k)| \to \infty$,
    \item[(ii)] There are $a > f^*$ and a subsequence $\{x^{k_j}\}$ such that $f(x^{k_j}) \leq a$.
\end{itemize}

If Case (i) holds, then Proposition \ref{prop:CSS} implies that $d(x^k,\tilde{x}) \to \infty$ for some fixed $\tilde{x} \in M$. Suppose now that Case (ii) holds. By Proposition \ref{prop:CSS}, the subsequence $\{x^{k_j}\}$ is bounded. Then, Theorem \ref{thm:cluster} implies that the sequence $\{x^k\}$ is bounded. Therefore, it follows from  Theorem \ref{thm:liminf}$(ii)$ that $d(x^k,S) \to 0$.
\end{proof}

\begin{remark} \label{Ocoercive} 
If, for some fixed $\tilde{x} \in M$, there exists a constant $A>0$ such that
\begin{equation*}
\lim_{d(x, \tilde{x}) \to \infty} \frac{f(x)}{d(x, \tilde{x})} \geq A,
\end{equation*}
then the solution set $S$ is nonempty and bounded. It then follows from Corollary \ref{boundedS} that either $d(x^k,\tilde{x}) \to \infty$ for some fixed $\tilde{x} \in M$, or $d(x^k, S) \to 0.$
\end{remark}
Although the next two results have appeared previously in the literature \cite{Ferreira03042019,Wang2018}, we include them here for completeness and provide alternative proofs that deepen understanding and strengthen the contribution of the present paper centered on Theorem \ref{thm:key}, which is instrumental to our analysis and has played an essential role in the development and organization of the paper.

\begin{theorem}\label{thm:TeoPolyak}
Let $S \neq \emptyset$ and suppose that \( \displaystyle\sum_{k=1}^{\infty} \lambda_k^2 <\infty \). Then, $x^k \to x^{*} \in S$.
\end{theorem}
\begin{proof}

Take $\overline{x} \in S$ and $j \in \mathbb{N}$. Since $ f(x^j) > f(\overline{x})$, it follows from continuity of $f$ that there exists $\delta_j > 0$ such that $f(u) < f(x^j)$ for all $u \in B[\overline{x} ,2\delta_j]$. In particular, $d(x^j, \overline{x}) > 2\delta_j$. Let $d_{j} = d(x^{j}, \overline{x})$ for all $j \in \mathbb{N}$. Applying Theorem \ref{thm:key} with $z=x^{j+1},~ x=x^j$ and $\lambda= \lambda_j$, we obtain
$$
\cosh (\kappa d_{j+1}) \leq  \cosh (\kappa d_j)\cosh (\kappa \lambda_j) - \sinh (\kappa \lambda_j)\sinh \left(\kappa \delta_j/2\right),
$$
Thus,
$$
\cosh(\kappa d_{j+1}) \leq  \cosh(\kappa d_j) \cosh (\kappa \lambda_j).
$$
For $k > j$, the above inequality implies that
$$   \cosh(\kappa d_k) \leq   \cosh(\kappa d_j) \prod_{i=j}^{k-1} \cosh (\kappa \lambda_i).   $$
On the other hand, since $\displaystyle\lim_{t \to 0}\displaystyle\frac{\cosh t -1}{t^2} = 1/2$ and \( \displaystyle\lim_{k \to \infty} \lambda_k = 0 \), there exists $k_0 \in \mathbb{N}$ such that $\cosh (\kappa \lambda_k) <  1 + \kappa^2\lambda_{k}^2$ for all $k > k_0$.
Thus,  for $k > j > k_0$ we have
\begin{eqnarray}\label{eq:ln}
\ln \cosh(\kappa d_k) &\leq& \ln \cosh(\kappa d_j)  + \displaystyle \sum_{i=j}^{k-1} \ln(1 + \kappa^2\lambda_{i}^2) \nonumber\\
&\leq& \ln \cosh(\kappa d_j)  + \kappa^2\displaystyle \sum_{i=j}^{k-1} \lambda_{i}^2\nonumber \\
&\leq& \ln \cosh(\kappa d_j)  + \kappa^2\displaystyle \sum_{i=j}^{\infty} \lambda_{i}^2.
\end{eqnarray}
Since $\sum_{k=1}^{\infty} \lambda_k^2 < \infty$, the sequence $\{x^k\}$ is bounded. By Theorem \ref{thm:liminf}$(ii)$, there exists a subsequence $\{x^{k_j}\}$ of $\{x^{k}\}$ such that $x^{k_j} \to x^{*} \in S$. We may assume without loss of generality that $\overline{x} = x^{*}$. From \eqref{eq:ln}, it follows that for $k > k_j > k_0$ we have
\begin{eqnarray*}
\cosh (\kappa d(x^k,x^{*})) 
& \leq & 
\cosh (\kappa d(x^{k_j},x^{*}))\exp\!\left( \kappa^2\sum_{i=k_j}^{\infty} \lambda_{i}^2\right).
\end{eqnarray*}
Since $\sum_{k=1}^{\infty} \lambda_k^2 < \infty$, letting $j \to \infty$ yields
\[
\sum_{i=k_j}^{\infty} \lambda_i^2 \to 0
\quad \text{and} \quad
d(x^{k_j},x^*) \to 0.
\]
By the continuity and strict monotonicity of the function
$t \mapsto \cosh t$ on $[0,\infty)$, it follows that
\[
d(x^k,x^*) \to 0,
\]
that is, $x^k \to x^* \in S$.
\end{proof}
\begin{theorem}\label{complexity}
Let $S \neq \emptyset$ and suppose that \( \displaystyle\sum_{k=1}^{\infty} \lambda_k^2 <\infty \). Then, there exist constants $A,B>0$ such that
$$
\min\{[f(x^k)-f(x^*)] : k =0,\dots, N\}\leq \frac{ \kappa A\sum_{k=0}^{N}\lambda_k^2 +Bd^2(x^*,x^0)}{\sum_{k=0}^{N}\lambda_k}.
$$
\end{theorem}
\begin{proof}
From Theorem \ref{thm:TeoPolyak}, we have $x^k \to x^* \in S$. Now, from Proposition \ref{prop:comparison}, it follows that
\begin{equation}\label{eq:C1}
\cosh (\kappa d_{k+1}) \leq \cosh (\kappa d_k)\cosh (\kappa \lambda_k) - \sinh (\kappa d_k)\sinh (\kappa \lambda_k) \cos \gamma_k,
\end{equation}
where $d_{k+1}=d(x^{k+1},{x}^*)$, $d_{k}=d(x^{k},{x}^*)$, $\lambda_k=d(x^{k+1}, x^k)$, $v^k=\exp_{x^k}^{-1}(x^*)$, and $
\cos \gamma_k=\displaystyle\frac{\langle s^k, v^k \rangle }{d_k}
$. From the convexity of $f$, we have $f(x^*)\geq f(x^k)+\langle g^k, v^k\rangle$. Hence,
\begin{equation}\label{eq:C2}
0\leq \frac{f(x^k)-f(x^*)}{d_k||g^k||}\leq \frac{\langle s^k, v^k \rangle }{d_k}=\cos \gamma_k.
\end{equation} 
Combining \eqref{eq:C1}, \eqref{eq:C2}, and the fact that $\displaystyle \cosh{t}-1\leq \frac{t}{2}\sinh{t}$ for all $t\geq 0$, we get 
$$
\cosh (\kappa d_{k+1})- \cosh (\kappa d_k) \leq \kappa\cosh (\kappa d_k)\sinh (\kappa \lambda_k) \nonumber \\
 \times \bigg[ \frac{\lambda_k }{2} -\frac{\tanh(\kappa d_k)}{(\kappa d_k)} \frac{[f(x^k)-f(x^*)]}{||g^k||}\bigg]. 
$$
From \cite[Lemma 3.1]{wang2015linear}, we have
\begin{equation}\label{eq:C4}
\frac{\kappa \sinh (\kappa d_k)\left(d_{k+1}^2- d_k^2\right)}{2d_k}\leq [\cosh (\kappa d_{k+1})- \cosh (\kappa d_k)].  
\end{equation}
From \eqref{eq:C4}, and the above inequality, we have that for all $k \in \mathbb{N}$,
\begin{equation}\label{eq:C5}
d_{k+1}^2- d_k^2 \leq \kappa\frac{\sinh (\kappa \lambda_k)}{(\kappa \lambda_k)} \bigg[ \lambda_k^2 \frac{(\kappa d_k)}{\tanh(\kappa d_k)} - 2\lambda_k\frac{[f(x^k)-f(x^*)]}{||g^k||}\bigg]. 
\end{equation}
On the other hand, note that the inequality $ \frac{t}{\tanh{t}} \leq 1 + t$ holds for all \( t \geq 0 \); thus, $\frac{(\kappa d_k)}{\tanh(\kappa d_k)} \leq 1+\kappa d_k$ for all $k$. The function \(t \mapsto \frac{\sinh{t}}{t} \), defined on \( (0, \infty) \), is increasing and bounded below by 1. Since $\lambda_k \to 0$, we obtain $ \frac{\sinh (\kappa \lambda_k)}{(\kappa \lambda_k)} \leq C_1$. Now, since \( \{x^k\} \) is bounded, we have $d_k=d(x^k, x^{*})< C_2$ and $||g^k||<C_3$ for all $k$. Hence, from \eqref{eq:C5} and upon renaming the constants, we deduce that for every \( k \in \mathbb{N} \), 
$$
\lambda_k[f(x^k)-f(x^*)] \leq \kappa A \lambda_k^2 +B(d_k^2- d_{k+1}^2).
$$
Summing the above inequality from $k = 0$ to $N$, we obtain
$$
\sum_{k=0}^{N}\lambda_k[f(x^k)-f(x^*)] \leq \kappa A\sum_{k=0}^{N}\lambda_k^2 +B(d_{0}^2-d_{N+1}^2).
$$
Therefore,
\begin{equation*}\min\{[f(x^k)-f(x^*)] : k =0,\dots, N\}\leq \frac{ \kappa A\sum_{k=0}^{N}\lambda_k^2 +Bd^2(x^0,x^*)}{\sum_{k=0}^{N}\lambda_k}.
\end{equation*}
\begin{equation*}
\end{equation*}
\end{proof}
\section{Conclusions}\label{sec:conclusions}
Our main contribution is a Riemannian counterpart of Shepilov's Euclidean convergence analysis for the subgradient method on manifolds with lower bounded curvature. It complements recent advances in Riemannian convex optimization and, in particular, provides a general tool that illustrates the underlying contraction behavior in the hyperbolic setting which will be useful for the future development of other first-order methods on manifolds.
The explicit example constructed in the Poincar\'e disk model (see Section \ref{sec:preli}) has an unbounded solution set $S$. More precisely, $S:=\arg\min f$ coincides with the points of the disk on the $x$-axis. In this example, if we apply the subgradient method starting from an initial point $x^0=q^0 i$ on the $y$-axis then
\[
x^{k+1}=\exp_{x^{k}}\bigl(\lambda_{k}s^{k}\bigr)=q^{k+1}i
\]
remains on the $y$-axis for all $k\in\mathbb{N}$. The sequence $\{x^k\}$ is bounded because $-\operatorname{grad}f(x^k)$ always points toward the origin, 
\[
d(x^{k+1},0)\le\max\{\lambda_k,d(x^k,0)\},
\]
and $\displaystyle\lim_{k\to\infty}\lambda_k=0$. By Theorem \ref{thm:liminf} we have $d(x^k,S)\to0$. Since $d(x^k,S)=d(x^k,0)$, it follows that $x^k\to0$. Thus, we obtain a situation in which the solution set $S$ is unbounded, yet the sequence $\{x^k\}$ remains bounded and converges entirely to an element of $S$. Although the objective is smooth, the example indicates directions for further theoretical refinement, even in the smooth setting.
This construction suggests a broader geometric question. In the
comparison arguments used throughout this paper, the constant $\kappa$ enters only through the geodesic triangles generated by the subgradient step
\[
        x^{k+1}=\exp_{x^k}(\lambda_k s^k),
        \qquad
        s^k=-\frac{g^k}{\|g^k\|}.
\]
Thus, instead of assuming a uniform lower bound
\[
        \sec M\geq -\kappa^2,
\]
one may ask whether the convergence analysis can be extended to Hadamard
manifolds whose curvature is not uniformly bounded below, but whose negative part grows at most linearly with the distance from a fixed reference point $\bar x\in M$.

More precisely, a natural condition is the existence of constants
$A,B\geq 0$ such that, for every $p\in M$ and every two-dimensional
subspace $\sigma\subset T_pM$,
\begin{equation}\tag{LC}\label{eq:LC}
K_p(\sigma)\geq -\bigl(A+B\,d(p,\bar x)\bigr).
\end{equation}

The case $B=0$ recovers the constant lower curvature bound considered in the
present work, with $\kappa=\sqrt{A}$. If a geodesic triangle associated with a
subgradient step is contained in a ball $B[\bar x,R]$, then condition
\eqref{eq:LC} implies the local estimate
\[
K_p(\sigma)\geq -(A+BR),
\]
in that region. Consequently, the same hyperbolic comparison estimates may be
applied locally with the effective parameter
\[
        \kappa_R:=\sqrt{A+BR}.
\]
This observation suggests a possible extension of the present analysis. The main difficulty is that the effective comparison parameter, \(\kappa_R\), depends on the size of the region visited by the iterates. Therefore, one would need a bootstrap argument showing that the curvature contribution generated along the orbit remains controlled before boundedness of the sequence is known. Such an approach would extend the current theory from manifolds with uniformly bounded below curvature to a class of Hadamard manifolds with linearly decaying lower curvature bounds, and will be investigated in future work.

\begin{acknowledgement}
The research of Glaydston de C. Bento, J. Xavier da Cruz Neto, and Jurandir de O. Lopes was partially supported by the Conselho Nacional de Desenvolvimento Científico e Tecnológico (CNPq) under grants 314106/2020-0, 302156/2022-4, and 305415/2025-5, respectively.
\end{acknowledgement}

\end{document}